\documentclass[a4paper,12pt]{amsart}
\usepackage{comment,amsmath,amssymb,amsthm,changepage,pifont,graphicx,tikz}
\usepackage[english]{babel}
\usepackage{todonotes}
\usetikzlibrary{patterns}

\newtheorem{theorem}{Theorem}

\newtheorem{lemma}[theorem]{Lemma}

\newtheorem{conjecture}[theorem]{Conjecture}

\newcommand{\FF}{\mathbf{F}}
\renewcommand{\AA}{\mathbf{A}}
\newcommand{\ZZ}{\mathbf{Z}}
\newcommand{\RR}{\mathbf{R}}
\newcommand{\QQ}{\mathbf{Q}}

\newcommand{\CC}{\mathbf{C}}

\DeclareMathOperator{\Tr}{Tr}

\DeclareMathOperator{\codim}{codim}
\DeclareMathOperator{\conv}{conv}
\DeclareMathOperator{\supp}{supp}

\def\cI{{\mathcal I}}
\def\cJ{{\mathcal J}}

\def\11{{\mathbf 1}}

\numberwithin{theorem}{subsection}

\DeclareMathOperator{\charac}{char}
\DeclareMathOperator{\ord}{ord}

\begin{document}

\title{New bounds for exponential sums with a non-degenerate phase polynomial}
\author{Wouter Castryck}
\address{imec-Cosic, Departement Elektrotechniek, KU Leuven\\
         Kasteelpark Arenberg 10/2452, 3001 Leuven, Belgium}
\email{wouter.castryck@esat.kuleuven.be}
\author{Kien Huu Nguyen}
\address{Laboratoire Paul Painlev\'e, Universit\'e de Lille-1\\
         Cit\'e Scientifique, 59655 Villeneuve d'Ascq Cedex, France}
\email{hkiensp@gmail.com}

\date{}

\begin{abstract} 
We prove a recent conjecture due to Cluckers and Veys on exponential sums modulo $p^m$ for $m \geq 2$ in the special case where the phase polynomial $f$ is sufficiently non-degenerate with respect to its Newton polyhedron at the origin. Our main auxiliary result is an
improved bound on certain related exponential sums over finite fields. This bound
can also be used to settle a conjecture of Denef and Hoornaert on the candidate-leading Taylor coefficient of Igusa's local zeta function associated to a non-degenerate polynomial, at its largest non-trivial real candidate pole.\\

\end{abstract}

\maketitle

\section{Introduction}

\subsection{} \label{introintro} Let $f \in \ZZ[x]$ be a non-zero polynomial in the variables $x = x_1, \ldots, x_n$ such that $f(0)=0$. In this article we prove new bounds on the absolute value of exponential sums
of the form
\begin{equation*} 
 S_f(p,m) := \frac{1}{p^{mn}} \sum_{x \in \{0, \ldots, p^m-1\}^n }   
\exp \left( 2 \pi \mathbf{i} \frac{f(x)}{p^m} \right)
\end{equation*}
where $p$ is a prime number and $m \geq 1$ is an integer. We work 
under the assumption that $f$ is non-degenerate with respect to the faces of its Newton polyhedron $\Delta_0(f)$ at the origin, in the strong sense recalled in Section~\ref{nondegrecall} below. 
Concretely, let $\sigma \in \QQ_{>0}$ be such that $(1/\sigma, \ldots, 1/\sigma)$ is contained in a proper face of $\Delta_0(f)$, and let $\kappa$ denote the maximal codimension in $\RR^n$ of such a face. Then we prove
the existence of a constant $c \in \RR_{>0}$ which only depends on $f$ such that
\begin{equation} \label{firstbound}
 | S_f(p,m) | \leq c p^{- \sigma m} m^{\kappa - 1}
\end{equation}
for all sufficiently large prime numbers $p$ and all integers $m \geq 2$. Moreover, if $f$ is supported on a hyperplane which does not contain the origin and which has a normal vector in $\RR_{\geq 0}^n$,
then we can include $m = 1$ in the foregoing statement.

It is known  
that $\sigma_0(f) = \min \{1, \sigma \}$ where $\sigma_0(f)$ denotes the log canonical threshold
$\sigma_0(f)$ of $f$ at the origin~\cite[\S9.3.C]{lazarsfeld}. So our work implies the existence of a constant $c \in \RR_{>0}$ such that
\begin{equation} \label{igusabound}
 | S_f(p,m) | \leq c p^{- \sigma_0(f) m} m^{n - 1}
\end{equation}
for all primes $p$ and integers $m \geq 2$. 
Likewise, if
the support of $f$
is contained in a hyperplane not passing through the origin and having a normal vector in $\RR_{\geq 0}^n$, then
 the same conclusion holds with $m \geq 1$.
The fact that we can write `all primes' rather than `all sufficiently large primes' follows from an observation due to Igusa, which implies the bound~\eqref{igusabound} for all primes $p$ but for some constant $c$ that is a priori allowed to depend on $p$; see e.g.~\cite[p.\,364]{denefbourbaki} and~\cite[p.\,78]{igusa}.

\subsection{} \label{introcontext} Let us give some context for these results. Igusa's conjecture on exponential sums~\cite[p.\,2]{igusa}
predicts that a bound of the form~\eqref{igusabound} should hold for all primes $p$ and all integers $m \geq 1$, regardless of any non-degeneracy assumption but under the condition that $f$ is a non-constant homogeneous polynomial. This is related to the integrability of certain functions over the ad\`eles, which in turn is connected with the validity of a generalized Poisson summation formula of Siegel-Weil type~\cite[Ch.\,4]{igusa}.
  Igusa's conjecture was proven in the non-degenerate case by
Denef and Sperber~\cite[Thm.\,1.2]{denefsperber} subject to a certain combinatorial constraint on $\Delta_0(f)$. This constraint was later removed by Cluckers~\cite[Thm.\,3.1]{cluckersTAMS}, who in fact
proved the bound~\eqref{firstbound} for all $m \geq 1$ and all non-constant \emph{quasi-}homogeneous non-degenerate polynomials $f$.
Cluckers' result naturally leads to the following strengthening of the statement of Igusa's conjecture:
\begin{conjecture}[Igusa, generalization due to Cluckers] \label{igusaconjecture}
For all non-constant quasi-homogeneous polynomials $f \in \ZZ[x]$ there
exists a constant $c \in \RR_{>0}$ such that the bound~\eqref{igusabound} holds for all primes $p$ and all integers $m \geq 1$.
\end{conjecture}
In a recent paper Cluckers and Veys predict~\cite[Conj.\,1.2]{cluckersveys} that assuming $m \geq 2$ should allow to drop the quasi-homogeneity condition from the statement of Igusa's conjecture. 
However, now it should be taken into account that there may exist points $\alpha = (\alpha_1, \ldots, \alpha_n) \in \CC^n$ at which the log canonical threshold $\sigma_\alpha(f)$ of the hypersurface $f(x) - f(\alpha) = 0$ at $\alpha$ is strictly smaller than $\sigma_0(f)$. This phenomenon does not occur in the quasi-homogeneous case.
\begin{conjecture}[Cluckers, Veys] \label{conj_cluckersveys}
For all non-constant polynomials $f \in \ZZ[x]$ there exists a constant $c \in \RR_{>0}$ such that
the bound~\eqref{igusabound} holds for all primes $p$ and all integers $m \geq 2$, provided that
we replace $\sigma_0(f)$ by $\min_{\alpha \in \CC^n} \sigma_\alpha(f)$.
\end{conjecture}
Our contribution to this topic is twofold. Firstly, our result confirms Cluckers and Veys' conjecture in the special case where $f$ is non-degenerate; in this case the foregoing concern is void since the minimal log canonical threshold is always realized by $\sigma_0(f)$. 
Secondly, we raise the question whether in Conjecture~\ref{igusaconjecture} the assumption that $f$ is non-constant and quasi-homogeneous can be relaxed to the 
condition  that
 $\supp f$ is contained in a hyperplane not passing through the origin and having a normal vector in $\RR_{\geq 0}^n$. Here as well, this is provided that
we replace $\sigma_0(f)$ by $\min_{\alpha \in \CC^n} \sigma_\alpha(f)$; we can in fact restrict to 
those $\alpha$ for which $f(\alpha) = 0$ by a version of Euler's identity. We give an affirmative answer in the non-degenerate case. 

\subsection{} 
For convenience we have stated Igusa's conjecture and the Cluckers--Veys conjecture in terms of the log canonical threshold. However, several other versions have been put forward which may, in some cases, predict sharper bounds. In these versions the log canonical threshold is replaced by the motivic oscillation index~\cite{cluckersimrn,cluckersveys}, by the complex oscillation index~\cite[\S13.1.5]{arnoldbook}, or by minus the largest non-trivial real pole of Igusa's local zeta function associated with $f$~\cite{igusa}. In fact, in the Cluckers--Veys conjecture for the log canonical threshold, it is not entirely clear whether the condition $m \geq 2$ is absolutely necessary. The reason for including it comes from the other versions, where it is unavoidable in general (e.g., for
$f = x^2y -x$ as explained in \cite[Ex.\,7.2]{cluckersimrn}).

\subsection{} \label{sect:mainthm} We work at the following level of generality.
We fix a number field $K \supseteq \QQ$, let $\ZZ_K$ denote its ring of integers, and consider a polynomial $f \in \ZZ_K[x]$, where as before $x$ abbreviates a list of $n \geq 1$ variables $x_1, \ldots, x_n$. For each non-zero prime ideal $\mathfrak{p} \subseteq \ZZ_K$ we consider the $\mathfrak{p}$-adic completion $K_\mathfrak{p}$ of $K$, along with  its ring of integers 
$\ZZ_\mathfrak{p} = \{ \, a \in K_\mathfrak{p} \, | \, \ord_\mathfrak{p}(a) \geq 0 \, \}$
and its residue field $\FF_\mathfrak{p} = \ZZ_\mathfrak{p} / \mathfrak{p} $, whose cardinality we denote by
$N\mathfrak{p}$.  We denote by $| \cdot |_\mathfrak{p} = (N\mathfrak{p})^{-\ord_\mathfrak{p} (\cdot)}$ the corresponding non-archimedean norm on $K_\mathfrak{p}$.

Let $p$ be the prime number below $\mathfrak{p}$, and consider 
the additive
character
\[ \psi_\mathfrak{p} : K_\mathfrak{p} \rightarrow \CC^\ast : a \mapsto \exp (2 \pi \mathbf{i} \Tr_{K_\mathfrak{p} / \QQ_p} (a))  \]
where $\exp(2 \pi \mathbf{i} \, \cdot\, )$
is evaluated on $p$-adic numbers as follows: 
given $a \in \QQ_p$ we let $a'$ be a representant inside $\ZZ[1/p]$ of the residue class of $a$ modulo the ring of $p$-adic integers $\ZZ_p$ and we define
$\exp (2 \pi \mathbf{i} a)$ as $\exp (2 \pi \mathbf{i} a')$. 
Then to each $y \in K_\mathfrak{p}$ we associate the integral
\[ S_{f,\mathfrak{p}}(y) := \int_{\ZZ_\mathfrak{p}^n} \psi_\mathfrak{p}(yf(x)) | dx| \]
where $|dx| = |dx_1 \wedge \ldots \wedge dx_n |$ denotes the Haar measure, normalized such that
the volume of $\ZZ_\mathfrak{p}^n$ is $1$. 
Notice that the sum $S_f(p,m)$
from Section~\ref{introintro} can be rewritten as
$S_{f, (p)}(p^{-m})$.

Our main result is as follows:

\begin{theorem} \label{maintheorem}
  Let $f \in \ZZ_K[x]$ be a non-zero polynomial such that $f(0)=0$ and assume that it is non-degenerate with respect to the faces of its Newton polyhedron $\Delta_0(f)$ at the origin. 
Let $\sigma \in \QQ_{>0}$ be such that $(1/\sigma, \ldots, 1/\sigma)$ is contained in a proper face of $\Delta_0(f)$, and let $\kappa$ denote the maximal codimension of such a face.  
  Then there exists a constant $c \in \RR_{>0}$ only depending on $\Delta_0(f)$ such that for all non-zero prime ideals $\mathfrak{p} \subseteq \ZZ_K$ for which $N\mathfrak{p}$ is sufficiently large and all $y \in K_\mathfrak{p}$ satisfying $\ord_\mathfrak{p}(y) \leq -2$, we have
\[ \left| S_{f, \mathfrak{p}}(y)  \right| \leq c | y |_\mathfrak{p}^{-\sigma} | \ord_\mathfrak{p} (y) |^{\kappa-1}. \]
If $\supp f$ is contained in a hyperplane not passing through the origin and having a normal vector in $\RR_{\geq 0}^n$, then moreover $c$ can be chosen such that the bound
also applies to all $y \in K_\mathfrak{p}$ for which $\ord_\mathfrak{p}(y) = -1$.
\end{theorem}
\noindent 
It is clear that Theorem~\ref{maintheorem} implies the claims made in Sections~\ref{introintro} and~\ref{introcontext}.

The condition that $N \mathfrak{p}$ is sufficiently large allows
us to assume that
 $f$ is non-degenerate at 
$\mathfrak{p}$, by which we mean that
\[ f_\mathfrak{p} := f \bmod \mathfrak{p} \in \FF_\mathfrak{p}[x] \]
is non-degenerate with respect to the faces of its Newton polygon $\Delta_0(f_\mathfrak{p})$ at the origin
and $\Delta_0(f_\mathfrak{p}) = \Delta_0(f)$.
If $f \in \ZZ_K[x] \subseteq K[x]$ is non-degenerate with respect to the faces of $\Delta_0(f)$ to start from,
then it is indeed non-degenerate at all but finitely many non-zero prime ideals $\mathfrak{p} \subseteq \ZZ_K$; this follows, for instance, from Hilbert's Nullstellensatz.

Although we avoid a detailed discussion, we note that our proof of Theorem~\ref{maintheorem} also applies to other global fields, such as the field of rational functions $K = \FF_q(t)$ over a finite field $\FF_q$. However, here we have the extra condition that $\charac \FF_q$ should 
not be contained in the set $P$ of bad primes associated with $\supp f$, appearing in the statement of Theorem~\ref{finitefieldexpsum} below.
 (More generally, the discussion from~\cite[\S9]{cluckersduke} in the context of Igusa's conjecture applies here, too.)

\subsection{} 
In Section~\ref{section_background} we will explain how Theorem~\ref{maintheorem} arises as a consequence of the following
finite field exponential sum estimate,
which is the central auxiliary result of this paper and which is proven in Section~\ref{sect:finitefieldexpsum}.
\begin{theorem} \label{finitefieldexpsum}
Let $\FF_q$ be a finite field with $q$ elements and let $f \in \FF_q[x]$ be non-degenerate with respect to the faces of its Newton polyhedron $\Delta_0(f)$ at the origin.
Suppose that $\supp f$
is contained in a hyperplane which does not contain the origin and which has a normal vector in $\RR_{\geq 0}^n$. 
Let $\sigma \in \QQ_{>0}$ be maximal such that $(1/\sigma, \ldots, 1/\sigma) \in \Delta_0(f)$ and let
$\varphi : \FF_q \rightarrow \CC^\ast$ be a non-trivial additive character on $\FF_q$.  There exist a constant $c \in \RR_{>0}$ which only depends on $\Delta_0(f)$ and a finite set of primes $P$ which only depends $\supp f$ such that 
\[ \left| \, \frac{1}{(q-1)^n} \sum_{x \in \FF_q^{\ast n}} \varphi(f(x)) \, \right| \, < \, c q^{-\sigma} \]
as soon as $p = \charac \FF_q \notin P$.
\end{theorem}
\noindent If the hyperplane can be chosen such that it has a normal vector in $\RR_{>0}^n$ then $f$ is quasi-homogeneous, in which case the foregoing result is due to Cluckers~\cite[Prop.\,6.2 \& its proof]{cluckersTAMS}. 

\subsection{} \label{denefhoornaertintro} Theorem~\ref{finitefieldexpsum} also implies  
a conjecture by Denef and Hoornaert, involving
Igusa's local zeta function
\[ Z_{f, \mathfrak{p}} : \{ s \in \CC \, | \, \Re s > 0 \} \rightarrow \CC : s \mapsto \int_{\ZZ_\mathfrak{p}^n} |f(x)|^s_\mathfrak{p} |dx|,  \]
which one can associate to any non-zero prime ideal $\mathfrak{p} \subseteq \ZZ_K$ and any polynomial $f \in \ZZ_K[x]$.
 It is well-known that this is a rational function in $N\mathfrak{p}^{-s}$, so it admits a meromorphic continuation 
 to the entire complex plane, where one may see poles showing up. These poles are believed to contain important arithmetic and geometric information about the hypersurface $f = 0$; see \cite{denefbourbaki,segers} for more background.
In our setting where $f$ is non-constant, vanishing at $0$, and non-degenerate with respect to the faces of its Newton polyhedron at the origin, Hoornaert proved~\cite[\S4.3]{hoornaert} that there is always at least one real pole and that the largest such pole is either $s = -\sigma$ or $s = -1$. If $s=-1$ is a pole, it is called trivial. The expected pole order of $-\sigma$ is $\kappa$, unless $-\sigma = -1$ in which case the expected pole order is $\kappa + 1$. Here `expected' means that the actual pole order is bounded by the said quantity and typically equals it; a sufficient condition for equality is $\sigma < 1$, but see~\cite[Thms.\,5.5, 5.17, 5.19]{denefhoornaert} and~\cite[Thm.\,4.10]{hoornaert} for more precise statements. 

In Section~\ref{sect:denef_hoornaert} we will demonstrate how Theorem~\ref{finitefieldexpsum} implies a certain uniformity in $\mathfrak{p}$ of (what is expected to be) the leading Taylor coefficient at $s = -\sigma$. More precisely, we will show that the real number
\begin{equation*} 
 \lim_{s \rightarrow -\sigma} (N\mathfrak{p}^{s + \sigma} - 1)^{\kappa + \delta_{\sigma, 1}} Z_{f,\mathfrak{p}}(s) 
\end{equation*}
 is in $O(N\mathfrak{p}^{1 - \max \{1,\sigma\}})$ as $\mathfrak{p}$ varies; here $\delta_{\cdot, \cdot}$ denotes the Kronecker delta. This was proved by Denef and Hoornaert~\cite[\S5]{denefhoornaert} under a certain combinatorial assumption on $\Delta_0(f)$ which they conjectured to be superfluous. Our work confirms their conjecture. 
 One important subtlety is that Denef and Hoornaert work under a considerably weaker notion of non-degeneracy than we do, so that their conjecture is a priori stronger. However as explained in Section~\ref{sect:weakernotion_consequences} this is not a concern: we will see that our a priori weaker conclusion easily implies the Denef--Hoornaert conjecture in its full strength. This will rely on some conclusions made in Section~\ref{nondegversions}, where we elaborate on the difference between both non-degeneracy notions.

\subsection*{Acknowledgements} 
We would like to thank Raf Cluckers for outlining the strategy followed in this paper and for several other helpful remarks. Part of this work was prepared while the first author was affiliated with the Universit\'e de Lille-1 and with the University of Ghent. This work was supported by the European Commission through the European Research Council under the FP7/2007-2013 programme with ERC Grant Agreement 615722 MOTMELSUM.

\section{Non-degenerate polynomials and the invariant $\sigma$} \label{section_background}

\subsection{} \label{nondegrecall} Let us recall what it means for a polynomial  to be non-degenerate with respect to the faces of its Newton polyhedron at the origin, while fixing some notation. Let $k$ be a field, which from Section~\ref{sect:red} on will be either a number field $K$ or a finite field $\FF_q$. Let 
\[ f = \sum_{i \in \ZZ_{\geq 0}^n} c_i x^i \in k[x]\setminus \{0\} \] 
where as before $x = x_1, \ldots, x_n$ and $x^i = x^{(i_1, \ldots, i_n)}$ abbreviates $x_1^{i_1} \cdots x_n^{i_n}$. The Newton polyhedron of $f$ at the origin is defined as
\[ \Delta_0(f) = \conv \supp f + \RR_{\geq 0}^n, \]
with $\supp f = \left\{ \left.  i \in \ZZ_{\geq 0}^n \, \right| \, c_i \neq 0 \right\}$ the support of $f$. 
For all non-empty faces $\tau\subseteq \Delta_0(f)$ of any dimension, ranging from vertices to $\Delta_0(f)$ itself, we write 
\[ f_\tau = \sum_{i \in \tau \cap \ZZ_{\geq 0}^n} c_i x^i \]
and we say that $f$ is non-degenerate with respect to $\tau$ if the system of equations
\[ \frac{\partial f_\tau}{\partial x_1} = \ldots = \frac{\partial f_\tau}{\partial x_n} = 0 \]
has no solutions in $\overline{k}^{\ast n}$, or in other words if the map
\[ \overline{k}^{\ast n} \rightarrow \overline{k} : \alpha \mapsto f_\tau(\alpha) \] 
has no critical values. Here $\overline{k}$ denotes an algebraic closure of $k$. Self-evidently we call $f$ non-degenerate with respect to the faces (resp.\ the compact faces) of $\Delta_0(f)$ if it is non-degenerate with respect to all choices (resp.\ all compact choices) of $\tau$.

\subsection{} \label{nondegversions} Our non-degeneracy notion arises naturally in the study of exponential sums, but is considerably stronger than some of its counterparts which can be found elsewhere in the existing literature. Most notably, often one merely imposes the generic condition that $0$ is not a critical value, or in other words that the hypersurface $f_\tau = 0$ has no \emph{singularities} in $\overline{k}^{\ast n}$, rather than critical points. 
We will refer to the latter notion of non-degeneracy as \emph{weak} non-degeneracy. (This is not intended to become standard terminology: the only purpose it serves is to avoid confusion throughout the remainder of this paper.) Clearly, if $f$ is non-degenerate with respect to a face $\tau \subseteq \Delta_0(f)$, then it is also weakly non-degenerate with respect to that face. An important remark is that the converse holds as soon as $\tau$ is contained in a hyperplane of the form
\[ H = \{ \, (i_1, \ldots, i_n) \, | \, c_1 i_1 + \ldots + c_n i_n = b \, \} \] 
with $c_1, \ldots, c_n, b \in \ZZ$ satisfying $\text{char} \, k \nmid b$. This can be seen using a weighted version of Euler's identity.
As a consequence, we could have equally well formulated Theorem~\ref{finitefieldexpsum} assuming weak non-degeneracy, rather than non-degeneracy.


\subsection{} It is convenient to introduce the following notation: to each vector $a = (a_1, \ldots, a_n) \in \RR_{\geq 0}^n$ we associate 
\begin{align*}
 \nu(a) & =  a_1 + \ldots + a_n, \qquad
 N(a) = \min_{x \in \Delta_0(f)} x \cdot a, \\
 F(a)  & = \{ \, x \in \Delta_0(f) \, | \, x \cdot a = N(a) \, \}.
\end{align*}
The latter set is a face of $\Delta_0(f)$ which is called the first meet locus of $a$. It is contained in the hyperplane  $a_1 x_1 + \ldots + a_n x_n = N(a)$. Every face arises as
the first meet locus of some vector, where we note that
$\Delta_0(f)$ itself is given by
$F(\vec{0}) = F((0, \ldots, 0))$. Similarly writing $\vec{1} = (1, \ldots, 1)$, we define 
\begin{align*}
\kappa & = \codim F(\vec{1}), \\
\sigma & = \nu(\vec{1})/N(\vec{1}) = n / N(\vec{1})
\end{align*}
as in the introduction; if $N(\vec{1}) = 0$ then we let $\sigma = + \infty$. Note that if $\sigma < + \infty$ then
$\sigma N(a) \leq \nu(a)$ for all $a \in \RR_{\geq 0}^n$. 
We will usually write $\sigma(f)$ rather than $\sigma$ to emphasize the dependence on $f$. It is natural to define $\sigma(0) = 0$.

\subsection{} For use in Section~\ref{sect:auxiliary} we prove the following list of properties of $\sigma$, which we believe to be interesting in their own right:

\begin{lemma} \label{sigma_props}
 Let $f \in k[x]$ and $g \in k[y]$ be polynomials vanishing at the origin, in the respective variables $x = x_1, \ldots, x_n$ and $y = y_1, \ldots, y_m$. Then:
 \begin{itemize}
   \item[(i)] $\sigma(f + g) = \sigma(f) + \sigma(g)$,
   \item[(ii)] $\sigma(fg) = \min \{ \sigma(f), \sigma(g) \}$,
   \item[(iii)] if $n \geq 2$ then $\sigma(f) \geq \sigma(f(x_1, \ldots, x_{n-1}, x_{n-1})$. 
 \end{itemize}
Here $f + g$ and $fg$ are viewed as polynomials in $k[x,y]$, while $f(x_1, \ldots, x_{n-1}, x_{n-1})$ is viewed as an element of $k[x_1, \ldots, x_{n-1}]$.
\end{lemma}

\noindent Specifying the ambient ring is actually not needed: viewing $f(x_1, \ldots, x_{n-1}, x_{n-1})$ as an element of $k[x_1, \ldots, x_n]$ does not change the corresponding value of $\sigma$.

\begin{proof}[Proof of the lemma]
We assume that $f,g \neq 0$, since otherwise the listed properties are trivial.

(i) If $P \in \partial\Delta_0(f)$ and $Q \in \partial \Delta_0(g)$, where $\partial$ denotes the topological boundary, then every convex combination of $(P,0)$ and $(0,Q)$ is contained in $\partial \Delta_0(f+g)$. Applying this to $P = (1/\sigma(f), \ldots, 1/\sigma(f)) \in \partial \Delta_0(f)$ and $Q = (1/\sigma(g), \ldots, 1/\sigma(g)) \in \partial \Delta_0(g)$ and considering the convex combination 
\[ \frac{\sigma(f)}{\sigma(f) + \sigma(g)} (P,0) + \frac{\sigma(g)}{\sigma(f) + \sigma(g)} (0,Q),\]
we see that the desired conclusion follows.

(ii) Assume without loss of generality that the minimum is realized by $\sigma(f)$. Note that 
\begin{equation*} \Delta_0(fg) = \Delta_0(f) \times
\Delta_0(g),
\end{equation*}
from which one sees that $1/ \sigma(fg) \geq 1/\sigma(f)$ or in other words that $\sigma(fg) \leq \sigma(f)$. For the converse inequality, we write $(1/\sigma(f), \ldots, 1/\sigma(f)) \in \RR_{\geq 0}^{n+m}$ as
\[ \left(\frac{1}{\sigma(f)}, \ldots, \frac{1}{\sigma(f)},  \frac{1}{\sigma(g)}, \ldots, \frac{1}{\sigma(g)}
 \right) + 
\left( 0, \ldots, 0,  \frac{1}{\sigma(f)} - \frac{1}{\sigma(g)}, \ldots, \frac{1}{\sigma(f)} - \frac{1}{\sigma(g)} \right), \]
which is seen to be an element of $\Delta_0(fg)$, proving that $\sigma(fg) \geq \sigma(f)$.

(iii) Write $f' = f(x_1, \ldots, x_{n-1}, x_{n-1})$. We claim that if $(i_1, \ldots, i_{n-2}, i) \in \Delta_0(f')$ then there exist $a,b \in \RR_{\geq 0}$ such that $(i_1, \ldots, i_{n-2}, a,b) \in \Delta_0(f)$ and $a + b = i$. Indeed, this property is clear for any point in the support of $f'$, and the claim follows by convexity considerations. Now apply this claim to the point
\[ \left( \frac{1}{\sigma(f')}, \ldots, \frac{1}{\sigma(f')} \right) \in \Delta_0(f') \]
to find $a,b \in \RR_{\geq 0}$ for which
 \[ \left( \frac{1}{\sigma(f')}, \ldots, \frac{1}{\sigma(f')}, a, b \right) \in \Delta_0(f) \]
with $a + b = 1/\sigma(f')$.  When adding $(0, \ldots, 0, 1/\sigma(f') - a, 1/\sigma(f') - b) \in \RR_{\geq 0}^n$ to this, we stay inside $\Delta_0(f)$, from which we conclude the desired inequality.
\end{proof}

 We remark that these properties are reminiscent of well-known facts on the log canonical threshold at the origin.
Indeed, 
with the notation and assumptions from above (and the extra assumption that $f, g \neq 0$), one has that
\begin{itemize}
  \item[(a)] $\sigma_0(f + g) = \min \{ 1, \sigma_0(f) + \sigma_0(g) \}$,
  \item[(b)] $\sigma_0(fg) = \min \{ \sigma_0(f), \sigma_0(g) \}$,
  \item[(c)] $\sigma_0(f) \geq \sigma_0(f(x_1, \ldots, x_{n-1}, x_{n-1}))$.
\end{itemize}
See~\cite[Rmk.\,2.10]{demaillykollar}, \cite[Prop.\,9.5.22]{lazarsfeld} and \cite[Thm.\,7.5 \& proof of Thm.\,8.20]{kollar}, respectively.
In fact, statements (i -- iii) could have also been settled as mere corollaries to (a -- c), by using the following trick. Notice that properties (i -- iii) are purely combinatorial, so we can replace $f$ and $g$ by non-degenerate polynomials over $\CC$ having the same Newton polyhedron; for what follows it suffices to assume weak non-degeneracy, which is generically satisfied. Next let
 $f_d = f(x_1^d, \ldots, x_n^d)$ and $g_d = g(y_1^d, \ldots, y_m^d)$ for some large positive integer $d$. These polynomials are again non-degenerate, and moreover $\Delta_0(f_d) = d \Delta_0(f)$ and $\Delta_0(g_d) = d \Delta_0(g)$ so that $\sigma(f) = d \sigma(f_d)$ and $\sigma(g) = d \sigma(g_d)$. If $d$ is large enough such that $\sigma(f_d) + \sigma(g_d) \leq 1$, then one sees that properties (i -- iii) follow from properties (a -- c) and the fact that $\sigma_0(h) = \min \{1, \sigma(h) \}$ for any polynomial $h$ that is weakly non-degenerate with respect to the faces of its Newton polyhedron at the origin~\cite[\S9.3.C]{lazarsfeld}.

\subsection{} We conclude with a bound on $\sigma(f)$ in terms of the dimension of the affine critical locus $C_f$ of $f$, i.e.\ the set of geometric points of $\AA^n_k$ at which all partial derivatives of $f$ vanish, where
we adopt the convention that the dimension of the empty scheme is $- \infty$.
If $k$ is a number field then the bound
can be viewed as a corollary to a general observation due to Cluckers: see~\cite[Thm.\,5.1]{cluckersimrn} and how this is applied in~\cite[Lem.\,6.3]{cluckersduke}, for instance. Our more direct approach has the advantage of working over any field.

\begin{lemma} \label{cluckerslemma}
 Let $f \in k[x] \setminus \{0\}$ be non-degenerate with respect to the faces of its Newton polyhedron $\Delta_0(f)$ at the origin, and denote by $\delta$ the dimension inside $\AA^n_k$ of $C_f$. Then $\sigma(f) \leq (n-\delta)/2$.
\end{lemma}

\begin{proof}
By non-degeneracy with respect to the entire Newton polyhedron $\Delta_0(f)$, every critical point of $f$ has at least one coordinate which is zero. Therefore it suffices to prove for all proper subsets $I$ of $\{1, \ldots, n\}$ that
\[ \sigma(f) \leq (n - \delta_I)/2 \]
where $\delta_I = \dim C_I$ with $C_I = \{ \, (a_1, \ldots, a_n) \in C_f \, | \, a_i \neq 0 \text{ if and only if } i \in I \, \}$. 

 By reordering the variables if needed we can assume that $I = \{ n - d + 1, \ldots, n \}$ for some integer $d$ which satisfies $0 \leq \delta_I \leq d < n$. If $C_I = \emptyset$ then there is nothing to prove. If $C_I \neq \emptyset$ then it contains at least one point $(a_1, \ldots, a_n)$, which by our assumption satisfies $a_1 = \ldots = a_{n-d} = 0$ and $a_{n-d+1}, \ldots, a_n \neq 0$.  We first claim that
\begin{equation} \label{polyhedron_winc}
  \Delta_0(f) \subseteq \Delta_0(x_1 + \ldots + x_{n-d}) \times 
  \RR_{\geq 0}^d.
\end{equation}
If we let $H$ denote the face of $\RR_{\geq 0}^n$ defined by 
\[ i_1 = \ldots = i_{n-d} = 0, \quad i_{n-d+1}, \ldots, i_n \geq 0 \]
and we define $\tau := \Delta_0(f) \cap H$,
then~\eqref{polyhedron_winc} is equivalent to $\tau = \emptyset$. Now suppose that $\tau  \neq \emptyset$: then it must concern a face of $\Delta_0(f)$. But $\partial f_\tau / \partial x_i$ vanishes identically for $i = 1, \ldots, n-d$, while it vanishes at
  $(1, \ldots, 1, a_{n-d+1}, \ldots, a_n)$ for $i = n-d + 1, \ldots, n$. This is in contradiction with the fact that $f$ is non-degenerate with respect to $\tau$, so our claim follows.

We now let $\ell \in \{0, \ldots, n-d\}$ be 
minimal such that
\[ \Delta_0(f) \subseteq \Delta_0(x_1 + \ldots + x_\ell + x_{\ell+1}^2 + \ldots + x_{n-d}^2) \times \RR_{\geq 0}^2 \]
up to reordering the variables $x_1, \ldots, x_{n-d}$. In other words $f$
can be written as
\[ f = x_1 \cdot g_1(x_{n-d+1}, \ldots, x_n) + \ldots + x_\ell \cdot g_\ell(x_{n-d+1}, \ldots, x_n) + \ldots \]
for non-zero polynomials $g_1, \ldots, g_\ell \in k[x_{n-d+1}, \ldots, x_n]$, where the last dots consist of terms that are at least quadratic in the variables $x_1, \ldots, x_{n-d}$.
It is easy to check that $\sigma(f) \leq (n-d+\ell)/2$, so it suffices to show that $\delta_I \leq d - \ell$. This is trivial if $\ell = 0$, so assume that $\ell \geq 1$. By taking partial derivatives one observes
that 
\[ C_I = \{ \, (0, \ldots, 0, a_{n-d+1}, \ldots, a_n) \, | \, (a_{n-d+1}, \ldots, a_n) \in S_I \, \} \]
with $S_I \subseteq \overline{k}^{\ast d}$ the scheme defined by $(g_1, \ldots, g_\ell)$. 
We will establish the desired bound on $\delta_I$ by proving that $S_I$ is either empty or a smooth complete intersection. 

By the Jacobian criterion 
this amounts to showing that for all $(a_{n-d+1}, \ldots, a_n) \in S_I$ the rows of the matrix
\[ J = \begin{pmatrix} \frac{\partial g_1}{\partial x_{n-d+1}}(a_{n-d+1}, \ldots, a_n) & \ldots & \frac{\partial g_1}{\partial x_n}(a_{n-d+1}, \ldots, a_n) \\ 
\vdots & \ddots & \vdots \\
\frac{\partial g_\ell}{\partial x_{n-d+1}}(a_{n-d+1}, \ldots, a_n)& \ldots & \frac{\partial g_\ell}{\partial x_n}(a_{n-d+1}, \ldots, a_n)\end{pmatrix} \]
are linearly independent. Suppose this is not the case, then there exists a vector
$(\alpha_1, \ldots, \alpha_\ell) \neq (0,\ldots, 0)$ such that
$(\alpha_1, \ldots, \alpha_\ell) \cdot J = \vec{0}$.
Assume without loss of generality that $\alpha_1, \ldots, \alpha_{\ell'} \neq 0$ and $\alpha_{\ell'+1} = \ldots = \alpha_\ell = 0$, where $0 < \ell' \leq \ell$. Now consider the face $\tau \subseteq \Delta_0(f)$ obtained by intersecting $\Delta_0(f)$ with
\[  H : i_1 + \ldots + i_{\ell'} = 1, i_{\ell'+1} = \ldots = i_\ell = i_{\ell + 1} = \ldots = i_{n-d} = 0, i_{n-d+1}, \ldots, i_n \geq 0. \]
Then
\[ f_\tau = x_1 \cdot g_1(x_{n-d+1}, \ldots, x_n) + \ldots + x_{\ell'} \cdot g_{\ell'}(x_{n-d+1}, \ldots, x_n) \]
and one verifies that $(\alpha_1, \ldots, \alpha_{\ell'}, 1, \ldots, 1, 1, \ldots, 1, a_{n-d+1}, \ldots, a_n) \in \overline{k}^{\ast n}$ is a common root of its partial derivatives: a contradiction with the non-degeneracy assumption with respect to $\tau$. 
 \end{proof}

\section{Reduction to estimating finite field exponential sums} \label{sect:red}

\subsection{} \label{rewriteintegral} In this section we explain how proving Theorem~\ref{maintheorem} reduces to
proving the bound on finite field exponential sums stated in Theorem~\ref{finitefieldexpsum}. This resorts to a well-known reasoning by Denef and Sperber~\cite[Prop.\,2.1]{denefsperber}, a slight generalization of which was elaborated by Cluckers~\cite[Prop.\,4.1]{cluckersTAMS}. 
The idea is to partition the integration domain $\ZZ_\mathfrak{p}^n$
according to all possible valuations, ignoring 
the zero-measure set of points $x$ in which $0$ appears as a coordinate:
\[ S_{f, \mathfrak{p}}(y) = \int_{\ZZ_\mathfrak{p}^n} \psi_\mathfrak{p}(yf(x)) | dx| = \sum_{a \in \ZZ_{\geq 0}^n} \int_{\substack{x \in \ZZ_\mathfrak{p}^n \\ \ord_\mathfrak{p}(x) = a}} \psi_\mathfrak{p}(yf(x)) | dx|. \]
Let $\pi_\mathfrak{p}$ be a uniformizing parameter of $\ZZ_\mathfrak{p}$; if $\mathfrak{p}$ is unramified then one can just take $\pi_\mathfrak{p} = p$.
By introducing new variables $u = (u_1, \ldots, u_n)$ through the substitution $x_j \leftarrow \pi_\mathfrak{p}^{a_j} u_j$, with $a_j$ the $j$th coordinate of $a$, we can rewrite the above sum as
\[ \sum_{a \in \ZZ_{\geq 0}^n} N\mathfrak{p}^{-\nu(a)} \int_{u \in \ZZ_\mathfrak{p}^{\ast n}} \psi_\mathfrak{p}(y \pi_\mathfrak{p}^{N(a)}(f_{F(a)}(u) + \pi_\mathfrak{p} (\ldots) )) | du| \]
where the expression $(\ldots)$ takes values in $\ZZ_\mathfrak{p}$.  As explained in Section~\ref{sect:mainthm} we can assume that $f$ is non-degenerate at $\mathfrak{p}$, in which case Hensel's lemma implies that the integral is zero whenever $\ord_\mathfrak{p}(y) + N(a) \leq -2$. On the other hand since $\psi_\mathfrak{p}$ is trivial on $\ZZ_\mathfrak{p}$, as soon as $\ord_\mathfrak{p}(y) + N(a) \geq 0$ the integral is just the measure of $\ZZ_\mathfrak{p}^{\ast n}$, which is $(1 - N\mathfrak{p}^{-1})^n$. The most interesting case is where $\ord_\mathfrak{p}(y) + N(a) = -1$, in which case the integral equals
\[ N \mathfrak{p}^{-n} \sum_{x \in \FF_\mathfrak{p}^{\ast n}} \varphi_y (\overline{f}_{F(a)} (x)) \]
where 
\[ \varphi_y : \FF_\mathfrak{p} \rightarrow \CC^\ast : x \bmod \mathfrak{p} \mapsto \psi_\mathfrak{p}(y \pi_\mathfrak{p}^{-\ord_\mathfrak{p}(y)-1} x)\] and $\overline{f}_{F(a)}$ denotes the reduction of $f_{F(a)}$ mod $\mathfrak{p}$. Note that
$\varphi_y$ is a well-defined non-trivial additive character on $\FF_\mathfrak{p}$.

We end up with
\[ S_{f, \mathfrak{p}}(y) = (1 - N\mathfrak{p}^{-1})^n  
\sum_{ \substack{ \tau \text{ face} \\ \text{of } \Delta_0(f)}}
\left( A_{\tau, \mathfrak{p}}(y) + 
\frac{B_{\tau, \mathfrak{p}}(y)}{(N\mathfrak{p} - 1)^n} 
\sum_{x \in \FF_\mathfrak{p}^{\ast n}} \varphi_y (\overline{f}_\tau (x))
\right)
\]
where
\[
A_{\tau, \mathfrak{p}}(y) =  
\sum_{ \substack{ a \in \ZZ_{\geq 0}^n \text{\,s.t.\,} F(a) = \tau \\ \text{and } N(a) \geq -\ord_\mathfrak{p}(y) }  }  N\mathfrak{p}^{-\nu(a)}
\quad  \text{and}
\quad 
B_{\tau, \mathfrak{p}}(y) = \sum_{ \substack{ a \in \ZZ_{\geq 0}^n \text{\,s.t.\,} F(a) = \tau \\ \text{and } N(a) = -\ord_\mathfrak{p}(y) - 1 }  }  N\mathfrak{p}^{-\nu(a)}.
 \]
Here a trivial but important remark is that 
$B_{\tau, \mathfrak{p}}(y) = 0$ 
as soon as $\ord_\mathfrak{p}(y) \leq -2$ and the affine span of $\tau$ passes through the origin, because in this case $N(a) = 0$ while $-\ord_\mathfrak{p}(y) -1 \geq 1$.

\subsection{} Using the estimates from~\cite[Prop.\,5.2]{cluckersTAMS} or~\cite[\S3.4]{denefsperber} one sees that there exists a constant $c \in \RR_{>0}$ such that
\[ A_{\tau, \mathfrak{p}}(y) \leq c |y|_\mathfrak{p}^{-\sigma} | \ord_\mathfrak{p}(y)|^{\kappa - 1} \quad \text{and} \quad
B_{\tau, \mathfrak{p}}(y) \leq c |y|_\mathfrak{p}^{-\sigma} N\mathfrak{p}^{\sigma(f_\tau)} |\ord_\mathfrak{p}(y) |^{\kappa -1} \]
for all choices of $\mathfrak{p}$ and $y$,
where we recall that $\sigma(f_\tau) \leq \sigma$ is the minimal rational number such that $(1/\sigma(f_\tau), \ldots, 1/ \sigma(f_\tau))$ is contained in $\Delta_0(f_\tau) = \tau + \RR_{\geq 0}^n$. 
Using these bounds one verifies that in order to prove Theorem~\ref{maintheorem} it suffices to show 
that for all faces $\tau \subseteq \Delta_0(f)$ there
exists a constant $c$ only depending on $\Delta_0(f)$ such that
\begin{equation} \label{neededffexpsum} \left| \, \frac{1}{(N\mathfrak{p} - 1)^n} 
\sum_{x \in \FF_\mathfrak{p}^{\ast n}} \varphi_y (\overline{f}_\tau (x)) \, \right| \, \leq \, c \cdot N\mathfrak{p}^{-\sigma(f_\tau)}. 
\end{equation}
for all non-zero prime ideals $\mathfrak{p}$ having a sufficiently large norm.
In fact, it suffices to establish this bound for the faces $\tau$ whose affine span does not contain the origin.
Indeed, if $\ord_\mathfrak{p}(y) \leq -2$ then this claim is straightforward in view of the remark concluding Section~\ref{rewriteintegral}. On the other hand, if $\ord_\mathfrak{p}(y) = -1$ and $\supp f$ is contained in a hyperplane $H$
which
does not contain the origin and which has a normal vector in $\RR_{\geq 0}^n$, then in order to conclude~\eqref{neededffexpsum} for a given face $\tau \subseteq \Delta_0(f)$, it suffices to prove it for the face $\tau \cap H$, whose affine span does not contain the origin.

\section{New bounds for finite field exponential sums} \label{sect:finitefieldexpsum}

\subsection{} \label{finiteintro} Thus we are left with proving Theorem~\ref{finitefieldexpsum}.
Let $H$ be a hyperplane as in the statement; 
we can assume it to be of the form
\[ H : c_1i_1 + \ldots + c_ni_n = b \]
for $b \in \ZZ_{> 0}$ and $c_1, \ldots, c_n \in \ZZ_{\geq 0}$. Without loss of generality we can order the variables such that $c_1, \ldots, c_{n - r} > 0$ and $c_{n - r + 1}, \ldots, c_n = 0$. For simplicity we choose $H$ such that
$r$ is maximal. We can assume that $r > 0$ because if 
$r = 0$ then $f$ is quasi-homogeneous, and as mentioned in this case Theorem~\ref{finitefieldexpsum} is due to Cluckers~\cite[Prop.\,6.2]{cluckersTAMS}. 
Define $P$ as the set of prime numbers dividing $b$ and assume throughout the rest of this section that $\charac \FF_q \notin P$. Note that $P$ clearly depends on $\supp f$ only. In fact it even depends on $\Delta_0(f)$ only, but the reason for writing $\supp f$ in the statement of Theorem~\ref{finitefieldexpsum} is that  $P$ will be enlarged in Section~\ref{genericnondeglemma}, in a way which could a priori depend on the specific configuration of $\supp f$. 

\subsection{} \label{boundreduction} Rename the variables $x_{n-r + 1}, \ldots, x_n$ as $z = z_1, \ldots, z_r$ and view $f$ as a quasi-homogeneous polynomial over $\FF_q[z]$ in the remaining variables.
Note that none of the weights $c_j$ can exceed $b$, otherwise it would be possible to remove the term $c_ji_j$, which contradicts the maximality of $r$. We order the indices such that $x_1, \ldots, x_s$ are the variables for
which the corresponding weights $c_1, \ldots, c_s$ are equal to $b$. These necessarily appear linearly with a non-zero coefficient in $\FF_q[z]$, again by the maximality of $r$. The variables $x_{s+1}, \ldots, x_{n-r}$ which have a smaller non-zero weight $c_i$ are renamed $y = y_1, \ldots, y_t$. Note that $n = s+r+t$.
Then we can write
$$f = h+g_{1}x_{1}+ \ldots +g_{s}x_{s},$$
where $g_1, \ldots, g_s \in \FF_q[z] \setminus \{0\}$ and where $h \in \FF_q[y,z]$
is quasi-homogeneous when considered over $\FF_q[z]$. Moreover for each  concrete value of $z \in \overline{\FF}_q^r$ the polynomial $h(\cdot, z) \in \overline{\FF}_q[y]$ admits $y = 0$ as a critical point. For the sake of clarity, we note that this includes the case where $h(\cdot, z)$ is identically zero, in which case every point is considered critical.

For each $I\subseteq \{1, \ldots ,s\}$ and $0\leq d\leq t$ we define  
$$V_{I,d}=\{ \, z\in \overline{\FF}_q^{\ast r} \,| \, \dim C_{h(\cdot, z)}=d, g_{i}(z)=0\Leftrightarrow i\in I \, \}$$
and we write
\begin{align*}
& \left|\dfrac{1}{(q-1)^{n}}\sum_{(x,y,z)\in \FF_q^{\ast n}} \varphi(f(x,y,z))\right|\\
&=\left|\sum_{I,d}\sum_{z\in V_{I,d}(\FF_q)}\dfrac{1}{(q-1)^{n}}[\sum_{x\in \FF_{q}^{\ast s}}\varphi(\sum_{i\notin I}g_{i}(z)x_{i})]\times [\sum_{y\in \FF_q^{\ast t}}\varphi(h(y,z))] \right|\\
&=\left|\sum_{I,d}\sum_{z\in V_{I,d}(\FF_q)}\dfrac{1}{(q-1)^{n}}(-1)^{s-\#I}(q-1)^{\#I}\times [\sum_{y\in \FF_q^{ \ast t}}\varphi(h(y,z))] \right|\\
&\leq \sum_{I,d} c_{I,d} q^{\dim V_{I,d}} (q-1)^{\# I - n} \left| \sum_{y\in \FF_q^{ \ast t}}\varphi(h(y,z)) \right|
\end{align*}
where we recall our convention that the dimension of the empty scheme is $-\infty$. In the last step we used the Lang--Weil 
estimates, which introduce constants $c_{I,d} \in \RR_{>0}$ that can be taken to depend on $I,d$ and $\Delta_0(f)$ only.

\subsection{} If $h$ is identically zero then $\dim C_{h(\cdot, z)} = t$ for all $z$ and the foregoing bound simplifies to
\begin{equation} \label{boundhis0} 
  \sum_I c_{I} q^{\dim V_I} (q-1)^{t + \#I - n}
\end{equation}
where $c_I = c_{I,t}$ and
\[ V_I = V_{I,t} = \{ \, z \in \overline{\FF}_q^{\ast r} \, | \, g_i(z) = 0 \Leftrightarrow i \in I \, \} . \] By the non-degeneracy of $f$, as in the proof of Lemma~\ref{cluckerslemma} we see that the scheme cut out by $(g_i)_{i \in I}$ is either empty or a smooth complete intersection. In particular $\dim V_I \leq r - \#I$. From this we see that~\eqref{boundhis0} is bounded by $cq^{-s}$ where $c = \sum_I c_I$. Now from
\[ \Delta_0(f) = \Delta_0(g_1x_1 + \ldots + g_sx_s) \subseteq \Delta_0(x_1 + \ldots + x_s) \times \RR_{\geq 0}^{t + r} \]
it is immediate that $\sigma(f) \leq s$, from which the desired bound follows.

\subsection{} \label{sect:auxiliary} Thus we can assume that $h$ is not identically zero.
In this case we estimate the bound from Section~\ref{boundreduction} by 
 \begin{align*}
&\sum_{I,d} c_{I,d} q^{\dim V_{I,d}} (q-1)^{\# I - n} a_d q^{\frac{t+d}{2}} \\
&\leq \sum_{I,d} c_{I,d}a_{d}q^{\dim V_{I,d} +\frac{t+d}{2}}(q-1)^{\#I-n}.
\end{align*}
Here we used an estimate due to Cluckers~\cite[Thm.\,7.4]{cluckersTAMS}, which introduces constants $a_d \in \RR_{>0}$ and excludes a finite set of field characteristics. 
From the proof of~\cite[Thm.\,7.4]{cluckersTAMS} and the references therein, mainly to~\cite[Cor.\,6.1]{cluckersduke} which in turn invokes~\cite[Thm.\,4]{katz}, it follows that $a_d$ can be taken to depend on $d$ and $\Delta_0(f)$ only. It also follows that the excluded field characteristics are 
contained in $P$, so this was
already accounted for (by our assumption that $\charac \FF_q \notin P$).

Therefore it suffices to prove that
for each $I,d$ we have 
\begin{equation} \label{1} 
\dim V_{I,d} +\frac{t+d}{2}+\#I-n\leq -\sigma(f).
\end{equation}
Now consider the following lemma: 
\begin{lemma}\label{2}
Let $I\subseteq \{1,...,s\}$ and define 
$$f_{I} = h + \sum_{i\in I} g_{i} x_{i} \in \FF_q[(x_i)_{i \in I}, y, z].$$
Then $\sigma(f) \leq \sigma(f_{I})+s-\#I$. In particular we have that $\sigma(f) \leq \sigma(h)+s$.
\end{lemma}
\begin{proof}
We proceed by induction on $s-\#I$.
Note that if $s - \#I =0$ then $f_{I}=f$ and there is nothing to prove. 
If $s - \#I \geq 1$ then consider an index $j$ which is not contained in $I$. By the induction hypothesis we know that $\sigma(f) \leq \sigma(f_{I \cup \{j\}}) + s - \#I - 1$. The lemma then follows by
\begin{align*}
  \sigma(f_{I \cup \{j\}}) & =  \sigma ( h + \sum_{i\in I} g_{i} x_{i} + g_j(z_1, \ldots, z_r)x_j  ) \\
  & \leq \sigma ( h + \sum_{i\in I} g_{i} x_{i} + g_j(u_1, \ldots, u_r)x_j ) \\
  & \leq \sigma(f_I) + \sigma(g_j(u_1, \ldots, u_r) x_j) \\
  & \leq \sigma(f_I) + 1,
\end{align*}
where $u_1, \ldots, u_r$ are new variables and the first two inequalities follow from properties (iii) resp.\ (i) stated in Lemma~\ref{sigma_props}.
\end{proof}
 From this we see that in order to prove \eqref{1}, it is sufficient to demonstrate the following estimate:
\begin{theorem}\label{5}
If $h$ is a non-zero polynomial then we have
$$\dim V_{I,d} +\frac{t+d}{2}+\#I-n\leq -\sigma(f_I)-s+\#I$$
for all $I \subseteq \{1, \ldots, s \}$ and $0 \leq d \leq t$.
\end{theorem}
\noindent The remainder of this section is devoted to proving this theorem. Note that the condition that $h$ is non-zero implies that $\sigma(f_I)<+\infty$ for each $I\subseteq \{1,\ldots ,s\}$.

\subsection{} \label{sect:weak} We first prove a weaker statement: 
\begin{lemma}[weak version of Theorem~\ref{5}] \label{weak5}
If $h$ is a non-zero polynomial then we have
$$\frac{\dim V_{I,d}}{2} +\frac{t+d}{2}+\#I-n\leq -\sigma(f_I)-s+\#I$$
for all $I \subseteq \{1, \ldots, s \}$ and $0 \leq d \leq t$.
\end{lemma}
\begin{proof} Consider the following algebraic subsets of $\AA^{\#I + t + r}$:
\begin{align*}
C_{f_I} : & \dfrac{\partial f_{I}}{\partial x_{i}}=\dfrac{\partial f_{I}}{\partial y_{j}}=\dfrac{\partial f_{I}}{\partial z_{\ell}}=0 \  \text{(i.e., the affine critical locus of $f_I$),}\\
W_{I,d} : & \dfrac{\partial f_{I}}{\partial x_{i}}=\dfrac{\partial f_{I}}{\partial y_{j}}=0, z \in \tilde{V}_{I,d}, \\
W_I : & \dfrac{\partial f_{I}}{\partial x_{i}}=\dfrac{\partial f_{I}}{\partial y_{j}}=0.
\end{align*}
Here $i \in I$, $1\leq j\leq t$ and $1\leq \ell\leq r$, and 
$$\tilde{V}_{I,d}=\{ \, z\in \AA^r \,| \, \dim C_{h(\cdot, z)}=d, g_{i}(z)=0\Leftrightarrow i\in I \, \} \supseteq V_{I,d}.$$
It is easy to verify that $W_{I,d}$ has dimension $\dim \tilde{V}_{I,d} +\#I+d$ and is contained in $W_I$. On the other hand $\dim C_{f_I} \geq \dim W_I - r$, so we see that 
$$\dim C_{f_I} \geq \dim \tilde{V}_{I,d} +\#I+d-r \geq \dim V_{I,d} + \# I + d - r.$$
Now because $f_I$ is non-degenerate with respect to the faces of its Newton polyhedron at the origin, by Lemma~\ref{cluckerslemma} we have
$$\dfrac{-r-t-\#I+\dim C_{f_I}}{2}\leq -\sigma(f_I).$$
Hence
$$\dfrac{-r-t-\#I+\dim V_{I,d}+\#I+d-r}{2}\leq -\sigma(f_I).$$
From this the lemma follows.
\end{proof}
 This implies that Theorem~\ref{5} is true if $\dim V_{I,d} = 0$. 
In fact we will prove Theorem~\ref{5} by induction on $\dim V_{I,d}$, so this
settles the base case. Note that the theorem is trivial if $\dim V_{I,d} = - \infty$.

\subsection{} \label{genericnondeglemma} The induction will rely on the following auxiliary lemma, which is the source of our enlargement of $P$, which we announced in Section~\ref{finiteintro}.
\begin{lemma} \label{3} 
Let $I \subseteq \{1, \ldots, s\}$. There exists a finite set of primes $P_\emph{comp}$ which only depends on $\supp f_I$, such that the following holds as soon as $\charac \FF_q \notin P_\emph{comp}$. 
For each $a \in \overline{\FF}_q$ let 
\[ f_{I,a} = f_I((x_i)_{i \in I}, y_1, \ldots, y_t, z_1,  \ldots, z_{r-1}, a) \] 
denote the polynomial obtained from $f_I$ by substituting $a$ for $z_r$. Then there exists a non-zero polynomial $\zeta \in \FF_q[z_r]$ such that for
all $a$ for which $\zeta(a) \neq 0$ we have that 
\begin{itemize}
\item[(i)] $\sigma(f_I) \leq \sigma(f_{I,a})$,
\item[(ii)] $g_i(z_1, \ldots, z_{r-1}, a)$ is not identically zero for each $i=1, \ldots, s$,
\item[(iii)] $f_{I,a}$ is non-degenerate with respect to the faces of its Newton polyhedron at the origin.
\end{itemize}
\end{lemma}


\begin{proof}
Consider the following variation on the above assertion: 
\begin{quote} \emph{There exists a non-zero polynomial $\zeta \in \FF_q[z_r]$ such that for all $a \in \overline{\FF}_q$ for which $\zeta(a) \neq 0$ we have that
\begin{itemize}
  \item[(a)] each non-zero coefficient $c(z_r)$ of $f_I$, when viewed as a polynomial in $(x_i)_{i \in I}, y, z_1, \ldots, z_{r-1}$ over $\FF_q[z_r]$, satisfies $c(a) \neq 0$,
  \item[(b)] $f_{I,a}$ is weakly non-degenerate with respect to the faces of $\Delta_0(f_{I,a})$. 
\end{itemize}}
\end{quote}
These two properties imply (i--iii). Indeed, it is obvious that (a) implies (ii), while it also ensures that the Newton polyhedron of $f_{I,a}$ at the origin equals the image of $\Delta_0(f_I)$ under the projection $\pi_r : \RR^{\#I + t + r} \rightarrow \RR^{\#I + t + r - 1}$
 corresponding to dropping the last coordinate. In particular we have that
 \[ \Delta_0(f_I) \subseteq \Delta_0(f_{I,a}) \times \RR_{\geq 0} \]
 which implies (i). Finally, note that $f_I$ is supported on the hyperplane $H_I \subseteq \RR^{\#I + t + r}$
 defined by
\[ \sum_{j \in I} c_ji_j + c_{s + 1} i_{s + 1} + \ldots + c_{s + t} i_{s + t} = b,  \]
so that
 $f_{I,a}$ is supported on the hyperplane $\pi_r(H_I) \subseteq \RR^{\#I + t + r - 1}$, which is defined by the same equation.
Therefore property (b) implies (iii), because it suffices to verify non-degeneracy with respect to the faces $\sigma \subseteq \Delta(f_{I,a})$ which are contained in $\pi_r(H_I)$, and as discussed in Section~\ref{nondegversions} the notions of non-degeneracy and weak non-degeneracy coincide with respect to such faces; here we used that $\charac \FF_q \notin P$.

Now it is easy to see that for each fixed choice of $\supp f_I$ the 
foregoing variation amounts to the validity of a first-order sentence. Therefore, if we can prove the direct analogue of this variation in characteristic zero, then by the compactness theorem (Robinson's principle) we know that it must be true in all characteristics outside a finite set $P_\text{comp}$, from which the desired conclusion follows.

So assume that we are working over a field $k$ of characteristic $0$. That is, we consider a polynomial $f = h + g_1x_1 + \ldots + g_sx_s \in k[x,y,z]$ satisfying the analogues of the properties listed in Section~\ref{finiteintro} and at the beginning of Section~\ref{boundreduction}. As before we let $f_I = h + \sum_{i\in I} g_ix_i$ and for each $a \in \overline{k}$ we write $f_{I,a}$ for the polynomial obtained from $f_I$ by substituting $a$ for $z_r$. 
Our task is to show that properties (a) and (b) hold for all but finitely many $a \in \overline{k}$. Since for (a) this is immediate, from now on we assume that property (a) is satisfied, so that $\Delta_0(f_{I,a}) = \pi_r (\Delta_0(f_I))$. 
Our task is to prove that $f_{I,a}$ is weakly non-degenerate with respect to its Newton polyhedron at the origin, except possibly for another finite number of values of $a$. We note that a very similar problem was recently tackled by Esterov, Lemahieu and Takeuchi~\cite[Prop.\,7.1]{lemahieu}, but their setting is more difficult since they assume weak non-degeneracy with respect to the compact faces, only.

 Every face $\sigma \subseteq \Delta_0(f_{I,a})$ arises as the projection along $\pi_r$ of a face $\tau \subseteq \Delta_0(f_I)$ of the same codimension; see Figure~\ref{figures} for some elementary examples. 
\begin{figure}[h]
\centering
\begin{tikzpicture}[scale=0.4]
  \draw[->] (0,0)--(5,0);
  \draw[->] (0,0)--(0,5);
  \draw[->] (0,0)--(-2,-3);
  \draw[white,pattern=north west lines, pattern color=gray] (-1,3.5)--(-1,1.5)--(-0.2,0.5)--(1,0.3)--(2,0.8)--(3,1.5)--(3,5)--(-1,3.5);
  \draw (-1,3.5)--(-1,1.5)--(-0.2,0.5)--(1,0.3)--(2,0.8)--(3,1.5)--(3,5);
  \draw[dotted](-1,3.5)--(-1,-1.5);
  \draw[dotted](3,1.5)--(3,0);
  \draw[dashed](-1,-1.5)--(3,0);
  \node at (1,-1.2) {$\sigma$};
  \node at (2,3) {$\tau$};
  \node at (4.5,0.5) {$y_1$};
  \node at (-2.6,-2.6) {$y_2$};
  \node at (-0.6,4.5) {$z_1$};   
\end{tikzpicture}
\qquad \qquad 
\begin{tikzpicture}[scale=0.4]
  \draw[->] (0,0)--(5,0);
  \draw[->] (0,0)--(0,5);
  \draw[->] (0,0)--(-2,-3);
  \draw[white,pattern=north west lines, pattern color=gray] (-0.5,-2.9)--(0.5,-1.4)--(1.5,2.1) -- (1.5,5) -- (-0.5,2) -- (-0.5,-2.9);
  \draw (-0.5,-2.9) -- (0.5,-1.4) -- (1.5,2.1) -- (1.5,5);
  \draw[dotted](1.5,0)--(1.5,2.1);
  \draw[dashed](-0.5,-3)--(1.5,0);
  \node at (1.5,-0.8) {$\sigma$};
  \node at (0.8,2) {$\tau$};
  \node at (4.5,0.5) {$y_1$};
  \node at (-2.6,-2.6) {$z_1$};
  \node at (-0.6,4.5) {$z_2$};   
\end{tikzpicture}
\caption{Two examples of a face $\tau$ along with its projection $\sigma$}
\label{figures}
\end{figure}
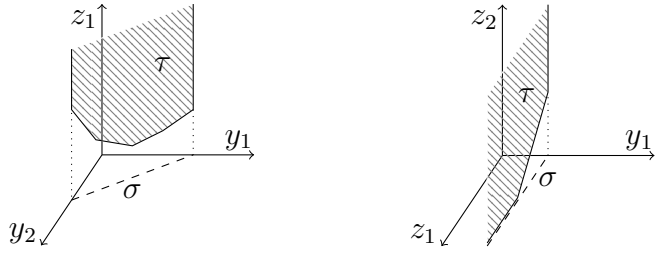
We will concentrate on the face $\sigma = \Delta_0(f_{I,a}) \cap \pi_r(H_{I})$, which is the projection of $\tau = \Delta_0(f_I) \cap H_I$. One verifies that $f_{I,a}$ is weakly non-degenerate with respect to $\sigma$ if and only if the fiber of the projection
\[ p_r : \{ \, ((x_i)_{i \in I}, y, z ) \in \overline{k}^{\ast \#I + t + r} \, | \, f_I((x_i)_{i \in I}, y, z) = 0 \, \} \rightarrow \overline{k} \]
on the last coordinate is smooth. 
Note that the source is smooth, thanks to the non-degeneracy of $f_I$ with respect to $\tau$.
By the generic smoothness theorem, which is only available in characteristic zero, this implies that all but finitely many fibers of $p_r$ are smooth. (More precisely, this conclusion follows by
applying~\cite[Cor.\,III.10.7]{hartshorne} to $p_r|_X$ for each irreducible component $X$ of the source.) We find that $f_{I,a}$ is weakly non-degenerate with respect to $\sigma$ for all but finitely many $a \in \overline{k}$. Repeating this argument for all other faces $\sigma \subseteq \Delta_0(f_{I,a})$ then leads to the desired conclusion.
\end{proof}

As announced in Section~\ref{finiteintro}, we now enlarge the set $P$ of excluded primes by adjoining the sets $P_\text{comp}$ arising from multiple applications of Lemma~\ref{3}. Indeed, the lemma will be applied for all possible choices of $I \subseteq \{1, \ldots, s \}$, but it will also be applied recursively to the polynomials $f_{I,a}$. Nevertheless, one easily sees that for a fixed choice of $\supp f$ there are only a finite number of supports appearing: therefore $P$ remains finite.

\subsection{} 
We are now ready to prove Theorem~\ref{5}.
As announced, we will proceed by induction on $\dim V_{I,d}$, where the base case
$\dim V_{I,d} = 0$ was taken care of in Section~\ref{sect:weak}. So assume that $\dim V_{I,d} \geq 1$ and that Theorem~\ref{5} holds for all smaller dimensions.

If $r = 1$ then necessarily $\dim V_{I,d} = 1$, and because all $g_i(z) = g_i(z_1)$ are non-zero we have that $I = \emptyset$. In particular $f_I = h$. Now by Lemma~\ref{3} we can find an $a \in V_{I,d}$ such that $f_{I,a}$ is non-degenerate with respect to the faces of its
Newton polyhedron at the origin, and such that $\sigma(f_I) \leq \sigma(f_{I,a})$. It is possible that we need to pick $a$ from $\overline{\FF}_q \setminus \FF_q$, but extending the coefficient field is of no harm to the statement we are trying to prove (or in other words, we can assume without loss of generality that $\# \FF_q$ exceeds $\deg \zeta$). Our choice of $a$ implies that $C_{f_{I,a}}$ has dimension $d$, so by Lemma~\ref{cluckerslemma} we have that
\[ \dfrac{-t + d}{2} \leq - \sigma(f_{I,a}) \leq \sigma(f_I), \]
which using $r= \dim V_{I,d} = 1$ and $\#I = 0$ implies that
\[ \dim V_{I,d} + \dfrac{t+d}{2} + \#I - n \leq -\sigma(f_I) - s - \#I,\]
as wanted.

If $r \geq 2$ then we can proceed similarly. Indeed,
since $\dim V_{I,d} \geq 1$ there exists at least one coordinate $z_j$ such that the image of the projection $\pi_{j} : V_{I,d} \to \AA^{1}$ onto the $z_j$-coordinate is Zariski dense; we can suppose that $j=r$. Choose $a \in \pi_{r}(V_{I,d})$, again over an extension of $\FF_q$ if needed, such that the fiber $V_{I,d,a}$ of $\pi_r$ over $a$ is of codimension $1$ in $V_{I,d}$. 
By Lemma~\ref{3} we can moreover assume that $f_{I,a}$ is non-degenerate with respect to the faces of its
Newton polyhedron at the origin and that $\sigma(f_I) \leq \sigma(f_{I,a})$.
By applying our induction hypothesis we find that
$$\dim V_{I,d,a} +\dfrac{t+d}{2}+\#I-n+1 \leq -\sigma(f_{I,a})-s+\#I.$$
But we know that $$\dim V_{I,d,a} =\dim V_{I,d}-1,$$
therefore 
$$\dim V_{I,d}+\dfrac{t+d}{2}+\#I-n\leq -\sigma(f_{I,a})-s+\#I\leq -\sigma(f_I)-s+\#I,$$
as desired.

\section{Denef and Hoornaert's conjecture} \label{sect:denef_hoornaert}

\subsection{} This final section is devoted to a proof of Denef and Hoornaert's conjecture:
\begin{theorem} \label{dhconjecture}
Let $f \in \ZZ_K[x]$ be a non-zero polynomial such that $f(0)=0$ and assume that
it is weakly non-degenerate with respect to the faces of its Newton polyhedron at the origin. Let $\sigma$ and $\kappa$ be as in the statement of Theorem~\ref{maintheorem}. Then
\[
 \lim_{s \rightarrow -\sigma} (N\mathfrak{p}^{s + \sigma} - 1)^{\kappa + \delta_{\sigma, 1}} Z_{f,\mathfrak{p}}(s) = O(N\mathfrak{p}^{1 - \max \{1,\sigma\}})
\]
as $\mathfrak{p}$ varies over all non-zero prime ideals of $\ZZ_K$.  
\end{theorem}
\noindent By excluding finitely many prime ideals we can assume that $f$ is non-degenerate at $\mathfrak{p}$. Under this assumption Denef and Hoornaert proved the explicit formula
\begin{equation*} 
  Z_{f, \mathfrak{p}}(s) = L_{\Delta_0(f)}(s) + \sum_{ \substack{ \tau \text{ proper face} \\ \text{of } \Delta_0(f)}} L_\tau(s) S_\tau(s)
\end{equation*}
where
\[ L_\tau(s) = N \mathfrak{p}^{-n} \left( (N\mathfrak{p} - 1)^n - \# \{ \, x \in \FF_\mathfrak{p}^\ast \, | \, \overline{f}_\tau(x) = 0 \, \} \cdot \frac{N\mathfrak{p}^{s+1} - N\mathfrak{p}}{N\mathfrak{p}^{s+1} - 1} \right) \]
and
\[ S_\tau(s) = \sum_{ \substack{ a \in \ZZ_{\geq 0}^n \\ \text{\,s.t.\,} F(a) = \tau  } }  N \mathfrak{p}^{-\nu(a) - N(a) s}, \]
see~\cite[Thm.\,4.2]{denefhoornaert}. Moreover they showed that
\[ \lim_{s \rightarrow - \sigma} (N\mathfrak{p}^{s + \sigma} - 1)^\kappa S_\tau(s) = 0\]
as soon as $\tau \not \subseteq \tau_0$, where
$\tau_0 \subseteq \Delta_0(f)$ denotes the smallest face containing $(1/\sigma, \ldots, 1/\sigma)$, see~\cite[Lem.\,5.4]{denefhoornaert}. Thus it suffices to prove that
\begin{equation} \label{dhreduction} \lim_{s \rightarrow -\sigma} (N\mathfrak{p}^{s + \sigma} - 1)^{\kappa + \delta_{\sigma, 1}} L_\tau(s) S_\tau(s) = O(N\mathfrak{p}^{1 - \max \{1,\sigma\}})
\end{equation}
for all subfaces $\tau$ of $\tau_0$. We remark that Denef and Hoornaert restict their discussion to $K = \QQ$, but everything readily generalizes to arbitrary number fields.

\subsection{} \label{sect:weakernotion_consequences} The face $\tau_0$ is contained in a hyperplane which does not contain the origin and which has a normal vector in $\RR_{\geq 0}^n$. Necessarily the same is true for all subfaces $\tau \subseteq \tau_0$. This has an important consequence related to the subtlety mentioned in Section~\ref{denefhoornaertintro}. Namely, Denef and Hoornaert work under the weak non-degeneracy assumption discussed in Section~\ref{nondegversions} (non-existence of singular points versus non-existence of critical points) and they also conjecture Theorem~\ref{dhconjecture} in terms of this weaker hypothesis. But
as mentioned at the end of Section~\ref{nondegversions}, over fields of large enough characteristic, both non-degeneracy notions coincide with respect to faces that are contained in a hyperplane not passing through the origin.
Therefore there is no ambiguity: proving~\eqref{dhreduction} under the assumption that $f_{\tau_0}$ satisfies our stronger non-degeneracy assumption is sufficient to conclude the Denef--Hoornaert conjecture in its full strength.

\subsection{} The proof works by explicit computation along the lines of Denef and Hoornaert, making a distinction between the cases $\sigma < 1$, $\sigma > 1$ and $\sigma = 1$. We make use of two new plug-ins. One of these plug-ins is our finite field exponential sum estimate stated in Theorem~\ref{finitefieldexpsum}, which implies:
\begin{lemma} \label{plugin1}
$ N\mathfrak{p} \cdot \# \{ \, x \in \FF_\mathfrak{p}^\ast \, | \, \overline{f}_\tau(x) = 0 \, \} = (N\mathfrak{p} - 1)^n + O(N\mathfrak{p}^{n + 1 - \sigma(f_\tau)})$.
\end{lemma}
\begin{proof}
Let $\varphi : \FF_\mathfrak{p} \rightarrow \CC^\ast$ be a non-trivial additive character, then
\[
  N\mathfrak{p} \cdot \# \{ \, x \in \FF_\mathfrak{p}^\ast \, | \, \overline{f}_\tau(x) = 0 \, \}  = (N\mathfrak{p}-1)^n + \sum_{t \in \FF_\mathfrak{p}^\ast} \sum_{x \in \FF_\mathfrak{p}^{\ast n}} \varphi(t\overline{f}_\tau(x))
\] 
and because $\tau$ is contained in a hyperplane which does not contain the origin and which has a normal vector in $\RR_{\geq 0}^n$, we can use Theorem~\ref{finitefieldexpsum} to find that
\[ \left| \sum_{t \in \FF_\mathfrak{p}^\ast} \sum_{x \in \FF_\mathfrak{p}^{\ast n}} \varphi(t\overline{f}_\tau(x)) \right| \leq \sum_{t \in \mathfrak{p}^\ast} c (N \mathfrak{p}-1)^n N\mathfrak{p}^{-\sigma(f_\tau)} \]
for some constant $c \in \RR_{> 0}$ that does not depend on $\mathfrak{p}$. The lemma follows.
\end{proof}

The second plug-in is a combinatorial inequality 
due to Cluckers~\cite[Thm.\,4.1]{cluckersduke}, which says that for any face $\tau \subseteq \Delta_0(f)$ and all $a \in \RR_{\geq 0}^n$ such that $F(a) = \tau$ we have $\nu(a) \geq \sigma N(a) + \sigma - \sigma(f_\tau)$. This leads to the following statement: 

\begin{lemma} \label{plugin2}
If $\tau \subseteq \tau_0$ then
\[ \lim_{s \rightarrow -\sigma} (N\mathfrak{p}^{s + \sigma} - 1)^\kappa S_\tau(s) = O(N\mathfrak{p}^{\sigma(f_\tau) - \sigma}). \] 
\end{lemma}

\begin{proof}
Denef and Hoornaert proved~\cite[Lem.\,5.4]{denefhoornaert} that
\[ \lim_{s \rightarrow -\sigma} (N\mathfrak{p}^{s + \sigma} - 1)^\kappa S_{\tau_0}(s) = c_0 \]
for some constant $c_0 \in \RR_{>0}$ which does not depend on $\mathfrak{p}$. Since $\sigma(f_{\tau_0}) = \sigma$ this settles the case $\tau = \tau_0$. If $\tau \subsetneq \tau_0$ then we can redo the proof of~\cite[Lem.\,5.16]{denefhoornaert}, but instead of invoking the Denef-Sperber inequality stated in~\cite[Lem.\,5.15]{denefhoornaert}
we use Cluckers' result. 
\end{proof}

\subsection{} We can now conclude. If $\sigma < 1$ then
\begin{equation} \label{fraction}
  \frac{N\mathfrak{p}^{-\sigma} - 1}{N\mathfrak{p}^{1-\sigma} - 1}
\end{equation}
is a negative quantity which is in $O(N\mathfrak{p}^{\sigma - 1})$. Together with Lemma~\ref{plugin1} this implies that
\[ L_\tau(-\sigma) = O(N\mathfrak{p}^{\sigma - \sigma(f_\tau)}),\]
which along with Lemma~\ref{plugin2} implies that
\[  \lim_{s \rightarrow -\sigma} (N\mathfrak{p}^{s + \sigma} - 1)^{\kappa} Z_{f,\mathfrak{p}}(s) = O(1) \]
as wanted.
If $\sigma > 1$ then~\eqref{fraction} becomes a positive quantity which equals $1 + o(1)$, leading to the estimate
\[ L_\tau(-\sigma) = O(N\mathfrak{p}^{1 - \sigma(f_\tau)}),\]
so here we find that
\[  \lim_{s \rightarrow -\sigma} (N\mathfrak{p}^{s + \sigma} - 1)^{\kappa} Z_{f,\mathfrak{p}}(s) = O(N\mathfrak{p}^{1 - \sigma}), \]
again as wanted.
Finally if $\sigma = 1$ then using Lemma~\ref{plugin1} one finds that
\[ \lim_{s \rightarrow -1} (N\mathfrak{p}^{s+1} - 1) L_\tau(s) = O(N\mathfrak{p}^{1 - \sigma(f_\tau)}) \]
which together with Lemma~\ref{plugin2} shows that
\[  \lim_{s \rightarrow -1} (N\mathfrak{p}^{s + 1} - 1)^{\kappa + 1} Z_{f,\mathfrak{p}}(s) = O(1), \]
thereby concluding the proof of Theorem~\ref{dhconjecture}.

\end{document}